\documentclass[12pt,a4paper,twoside]{article}

\usepackage{amsmath,amsfonts,amsthm,amsopn,color,amssymb,enumitem}
\usepackage{palatino}
\usepackage{graphicx}

\newcommand{\eps}{\varepsilon}

\newcommand{\R}{\mathbb{R}}

\newcommand{\RN}{{\mathbb{R}^N}}

\DeclareMathOperator{\meas}{meas}

\renewcommand{\le}{\leqslant}
\renewcommand{\ge}{\geqslant}
\renewcommand{\a }{\alpha }

\renewcommand{\d }{\delta }

\newcommand{\g }{\gamma }

\renewcommand{\l }{\lambda}
\newcommand{\n }{\nabla }

\renewcommand{\O}{\Omega}

\newcommand{\G}{\Gamma}

\newcommand{\A}{{\cal A}}

\newcommand{\D }{{\mathcal D}^{1,2}(\RN)}

\def\bbm[#1]{\mbox{\boldmath $#1$}}

\newtheorem{theorem}{Theorem}[section]
\newtheorem{lemma}[theorem]{Lemma}

\newtheorem{proposition}[theorem]{Proposition}
\newtheorem{remark}[theorem]{Remark}

\renewenvironment{proof}{\noindent{\textit{Proof.~}}}{$\hfill\square$\vspace{0.2 cm}\\}

\newenvironment{proofmain}{\noindent{\textit{Proof of Theorem \ref{th:main1}.~}}}{$\hfill\square$\vspace{0.2 cm}\\}
\newenvironment{proofmain2}{\noindent{\textit{Proof of Theorem \ref{th:main2}.~}}}{$\hfill\square$\vspace{0.2 cm}\\}

\title{{\bf On a prescribed mean curvature equation in Lorentz-Minkowski space \footnote{The author is supported  by GNAMPA Project ``Metodi Variazionali e Problemi Ellittici Non Lineari''}}}

\author{A. Azzollini \thanks{Dipartimento di Matematica, Informatica ed Economia, Universit\`a degli
Studi della Basilicata,  Via dell'Ateneo Lucano 10, I-85100
Potenza, Italy, e-mail: {\tt antonio.azzollini@unibas.it}}}

\date{}

\begin{document}
\maketitle

\begin{abstract}
    We are interested in providing new results on the following prescribed mean curvature equation in Lorentz-Minkowski space
        $$\n \cdot \left[\frac{\n u}{\sqrt{1-|\n u|^2}}\right] + u^p = 0,$$
    set in the whole $\RN$, with $N\ge 3$.\\
    The operator involved in the equation recurs in some questions related with classical relativity, for instance in determining maximal curvature hypersurfaces, or in the study of Born-Infeld theory of electrodynamics.\\
    We study both existence and multiplicity of radial ground state (namely positive and vanishing at infinity) solutions for $p>1$, emphasizing the fundamental difference between the subcritical and the supercritical case.\\
    We also study speed decay at infinity of ground states, and give some decay estimates.\\
    Finally we provide a multiplicity result on the existence of sign-changing bound state solutions for any $p>1$.
\end{abstract}

\section*{Introduction}

In this paper we are mainly interested in finding radial solutions for the problem
    \begin{equation}
\left\{
\begin{array}{l}
\n \cdot \left[\frac{\n u}{\sqrt{1-|\n u|^2}}\right] + u^p = 0,  \label{eq:mean}\tag{${\cal P}_+$}
\\
u(x)>0,\quad \hbox{in }\R^N,\\
u(x) \to 0 , \quad \hbox{as }|x|\to
\infty,
\end{array}
\right.
\end{equation}
where $N\ge 3$ and $1<p$.\\
The equation at the first line is quasilinear and involves the so called mean curvature operator in the Lorentz-Minkowski space which has been object of investigation in some recent papers.\\
The Euclidean version of the problem, where our equation is replaced by
$$\n \cdot \left[\frac{\n u }{\sqrt{1+|\n u|^2}}\right]+u^p=0,$$
has been studied for example by Ni and Serrin \cite{NS} and del Pino and Guerra \cite{DG} (see also the references therein). In those papers multiplicity and non existence results have been proved, depending on the choice of $p$.\\
At present, the literature concerning our equation is quite poor and mainly focused on the problem of finding positive solutions satisfying Dirichlet boundary conditions in bounded domains (see \cite{BJM, BJT, BJT2, CCR, COOR, COOR2}). In unbounded domains, and in particular in the whole $\RN$, equations involving mean curvature operator in Lorentz-Minkowski space are almost unexplored even if they have a considerable appeal from both physical and mathematical point of view (we refer to \cite{DRT} and the references therein). \\
In particular, we recall the strict relation between the equation we treat and the Born-Infeld (B-I for short) model in the theory of nonlinear electrodynamics. Assuming, in a static setting, that the magnetic field ${\bf H}={\n\times\bf A}$ is everywhere null and expressing the electric field as ${\bf E}=-\n u$, the B-I Lagrangian displays
    \begin{equation*}
        \mathcal L =b^2\left(1-\sqrt{1-\frac{|\n u|^2}{b^2}}\right),
    \end{equation*}
and the corresponding Euler-Lagrange equation is $\n \cdot \left[\frac{\n u}{\sqrt{b^2-|\n u|^2}}\right] = 0$. In the same spirit of \cite{BI}, Benci and Fortunato \cite{BF} proposed to describe the charged particles electrodynamics replacing B-I Lagrangian with the Maxwell one, and preserving the nonlinear structure by adding a perturbation $W(\sigma)$, where $\sigma = |{\bf A}|^2- |u|^2$ is a Poincar\'e invariant which makes the theory they developed consistent with general relativity.\\
Even if our equation is, in some sense, the effect of a sort of combination of the two theories, since it arises perturbing the electrostatic B-I Lagrangian with a pure power nonlinearity (we just assume $b=1$ for convenience), we remark that our study does not pursue the same physical purpose as \cite{BF} and \cite{BI}. Indeed, as observed in \cite{BF}, solutions of problem \eqref{eq:mean} have negative energy and then they are not suitable to represent charged particles (in general relativity energy corresponds to mass and then it must be positive). \\
Our study aims to add some new results to the work by Bonheure, De Coster and Derlet \cite{BDD}, where problem \eqref{eq:mean} was firstly studied. There they proved that if $p>\frac{N+2}{N-2}:=2^*-1$ (the so called supercritical case), than there exists at least one solution for \eqref{eq:mean} and there are infinitely many solutions to the equation, vanishing at infinity, but with no information on the sign. The authors exploited a very nice trick, consisting in truncating suitably the volume integral $\int_{\RN}(1-\sqrt{1-|\n u|^2})$ and then connecting problem \eqref{eq:mean} with that of finding minimizers for a constrained $C^1$ functional. \\
In this paper we wish to answer some questions left as open problems in \cite{BDD}, such as those concerning the existence of solution to \eqref{eq:mean} in the subcritical case (namely $1<p<2^*-1$) and the uniqueness of solution to \eqref{eq:mean} in the supercritical case.\\
The main results we provide are the following
	\begin{theorem}\label{th:main1}
		Assume that $1<p<2^*-1$. Then problem \eqref{eq:mean} has no radial solution.
	\end{theorem}
	\begin{theorem}\label{th:main2}
		Assume that $p>2^*-1$. Then problem \eqref{eq:mean} has infinitely many (radial) solutions not belonging to $\D$.
	\end{theorem}
Differently from \cite{BDD}, our approach to the equation is not variational. We reduce the PDE to an ODE by looking only for radial solutions of \eqref{eq:mean}, then we study the related Cauchy problem analyzing the behaviour of the solution in relation with the choice of the initial datum in $\R_+$. To be more explicit, for any $\xi>0$ we will consider in $\R_+$ the problem
\begin{equation}\label{cauchy}\tag{${\cal C}$}
\left\{
\begin{array}{ll}
\left(\frac{u'}{\sqrt{1-(u')^2}}\right)'+\frac{N-1}{r}\frac{u'}{\sqrt{1-(u')^2}}+|u|^{p-1}u=0
\\
u'(0) =0,\\
u(0)=\xi
\end{array}
\right.
\end{equation}
and we will look for global positive solutions. Some not so hard computations show that this type of solutions vanishes at infinity, so they are connected by a one-to-one correspondence to radial solutions of \eqref{eq:mean}.

According to a classical definition (see for example \cite{BLP, DG, PS} ), in the sequel we will call {\it ground states} the solutions to \eqref{eq:mean}. For completeness we recall that, up to our knowledge, besides \cite{BDD} the problem of existence of ground state solutions for equations involving our operator has been treated only in \cite{A}, when nonlinearity $u^p$ is replaced by something behaving, for instance, like $-\l u + u^p$ and in \cite{BDP}, where equation
$$-\n\cdot\left(\displaystyle\frac{\nabla \phi}{\sqrt{1-|\nabla \phi|^2}}\right)= \rho$$ 
is considered for $\rho$ corresponding to an assigned extended charge density or a superposition of deltas.

For both our results we take advantage of a very useful identity found by Erbe and Tang \cite{ET} and generalized in \cite{PuSe} by Pucci and Serrin and of a ``intersection point theorem" modeled on a similar one due to Franchi, Lanconelli and Serrin \cite{FLS}.\\
As we will explain better later, the Erbe-Tang identity makes clear the crucial role played by the critical value $2^*$ when we study our problem in $\R^N$. On the other hand, it is known from \cite[Example 5]{BJT} that, assuming Dirichlet conditions on boundary for the same problem set in a ball, existence of positive radial solution holds independently from the value of $p>1$, just requiring that the ball has a sufficiently large radius.\\
We emphasize the fact that the results we are going to prove agree with what we expect comparing with the analogous results concerning existence and multiplicity of ground states for the Lane-Emden equation
    \begin{equation}\label{eq:LE}
    -\Delta u = u^p.
    \end{equation}
However some remarks are in order.

We point out that our study is restricted to
radial functions, so that the nonexistence result we get for $p<2^*-1$ is not so strong as that obtained for Lane-Emden equation in \cite[Theorem 1.1]{GS}. Precisely, we are not able to exclude the existence of nonradial ground states, since at this time no radial symmetry result is available for solutions of \eqref{eq:mean}. \\
In the proof of Theorem \ref{th:main1} we take advantage of the study made in \cite{PuSe} on the number of points where the graphs of two radial ground states related with an ordinary quasilinear equation intersect. The arguments we apply are similar to those usually used to prove uniqueness of positive solution theorems. As a byproduct of Theorem \ref{th:main1} we deduce that every global solution of \eqref{cauchy} with $p<2^*-1$, changes sign.

When we deal with Lane-Emden equation, it is easy to verify that, exploiting the invariance with respect to suitably rescaled solutions, if we assume a radial ground state $u_1$ verifying $u_1(0)=1$ exists, then we obtain ground states such that $u(0)=\xi$ for arbitrary $\xi\in\R_+$ simply setting $u(|x|) =\xi u_1(\xi^{(p-1)/2}|x|)$. Thus we deduce that uniqueness of radial ground state never holds since either there isn't any (this is the case when $p<2^*-1$), or they are infinitely many (case occurring for $2^*-1\le p$).\\
On the contrary, in our situation it is soon seen that the structure of the equation destroys the invariance with respect to any type of inside/outside rescaling, so that we can not deduce multiplicity for $p>2^*-1$ just from the existence proved in \cite{BDD}. Moreover, sign-changing solutions of \eqref{cauchy} are present also in the case $p>2^*-1$ (Example 5 in \cite{BJT}) and this is a significant difference with respect to the Cauchy problem related with the ODE radial formulation of Lane-Emden.\\
In order to prove Theorem \ref{th:main2}, we will develop a different argument based on the comparison between the behaviour of solutions of \eqref{cauchy} and that of the solutions of the Cauchy problem releted with the ODE radial formulation of Lane-Emden.\\
We will show in a more detailed way in the sequel that, roughly speaking, the smaller initial datum $\xi$ in \eqref{cauchy} is, the more similar the behaviour of the solution of \eqref{cauchy} is with respect to that coming from the same Cauchy problem with the Lane-Emden equation. For this reason we understand why, when $p>2^*-1$, we observe the presence of ground states just for small values of $\xi$ whereas from a $\bar \xi$ on, all global solutions of \eqref{cauchy} are sign-changing.\\
We complete the study on radial ground state solutions comparing them with the one found by Bonheure, Derlet and De Coster in \cite{BDD}. We underline that, as stated in the final part in Theorem \ref{th:main2}, none of ground states we are finding corresponds to that found in \cite{BDD}, being this latter in $\D$. From a more precise analysis of the asymptotic behaviour at infinity we deduce the following result
	\begin{theorem}\label{th:main3}
		Let $p>2^*-1$ and $u$ be a radial solution to \eqref{eq:mean}. Then one of the following possibilities holds
			\begin{enumerate}
				\item $u\in\D$ and $u(x)=O(1/|x|^{N-2})$ for $|x|\to+\infty$;\\
				\item $u\notin\D$ and there exist $c_1,$ $c_2>0$ such that $c_1/|x|^{\frac{2N}{(N-1)(p+1)-2N}}\le u(x)\le c_2/|x|^{\frac{2}{p-1}}$ definitely for $|x|\to+\infty.$\\ Moreover there exists no $\a>2/(p-1)$ such that, definitely, $u(r)\le c/r^\a$ for some $c>0.$
			\end{enumerate}
	\end{theorem}

Finally, our focus shifts to global sign-changing solutions of \eqref{cauchy}. The presence both in the subcritical and in the supercritical case of this type of solutions of \eqref{cauchy} justifies a deeper study on their asymptotic behaviour at infinity. In particular, we are interested in looking for the existence of the so called {\it bound state solutions}, namely those solutions of a partial differential equation which vanish at infinity. Since we do not know anything about the sign of solutions found in the multiplicity theorem proved in \cite{BDD} and, on the other hand, we do not know if the solutions derived from \cite{BJT} and extended in the whole $\R_+$ go to zero as $r$ goes to infinity, the following result on the existence of sign-changing bound state solutions is completely new
\begin{theorem}\label{th:main4}
   Every sign-changing solution of \eqref{cauchy} is global and vanishes at infinity.
   In particular the problem
        \begin{equation}
            \left\{
                \begin{array}{l}
                    \n \cdot \left[\frac{\n u}{\sqrt{1-|\n u|^2}}\right] + |u|^{p-1}u = 0,  \label{eq:meanpm}\tag{${\cal P}_\pm$}
                    \\
                    u^{\pm}(x)\neq 0,\quad \hbox{in }\R^N,\\
                    u(x) \to 0 , \quad \hbox{as }|x|\to \infty,
                \end{array}
            \right.
        \end{equation}
   has infinitely many (radial) solutions for any $p>1$.
\end{theorem}

After discussing in section \ref{se:cauchy} some interesting properties of our solutions by studying the Cauchy problem they solve, in the succeeding sections we will prove our results, following this schema: in section \ref{se:grst} we look for solutions to \eqref{eq:mean} showing nonexistence and multiplicity theorems; in section \ref{se:dec} we prove Theorem \ref{th:main3} after a preliminary analysis on the differences arising when we compare our solutions to that in \cite{BDD}; in section \ref{se:signch}, finally, we perform an analytic study of the sign-changing radial solutions asymptotic behaviour to prove Theorem \ref{th:main4}.

In the sequel we shall use the standard notations and definitions for the Lebesgue and Sobolev spaces endowed with their usual norms.\\
Moreover we will call $c$ a positive constant suitably changing from line to line.\\
We finally point out that, since we are considering only functions with radial symmetry, we will use equivalently the notations $u(x)$ and $u(r)$ respectively for $x\in\RN$ and $r>0$ such that $|x|=r$.

\section{The Cauchy problem}\label{se:cauchy}
	It is well known that, radial solutions of the equation in problem \eqref{eq:meanpm}
can be found looking for solutions of \eqref{cauchy} and replacing the variable $r$ with $|x|$. Thanks to the local lipschitzianity of the pure power function, we classically determine local solutions $u_\xi$ for any $\xi>0$, namely solutions defined in a maximal interval $[0,R_\xi[$. \\
In order to simplify the writing, we will introduce the following notations inherited from \cite{PuSe}.\\
We set $\rho(r)=|u'(r)|$ and define the following functions in $C^1([0,1[,\R)$
	\begin{align*}
		A(\rho(r))&= 1/\sqrt{1-\rho^2(r)},\\ \O(\rho(r))&=\rho(r)/\sqrt{1-\rho^2(r)},\\
G(\rho(r))&=\int_0^{\rho(r)}\frac{t}{\sqrt{1-t^2}}\,dt=1-\sqrt{1-\rho^2(r)}
	\end{align*}
and
$$f(t)=|t|^{p-1}t\quad F(t)=\int_0^t f(s)\,ds=\frac 1{p+1}|t|^{p+1}$$
in $\R$.\\
Multiplying
	\begin{equation}\label{eq:eq}
		 \left(\frac{u'}{\sqrt{1-(u')^2}}\right)'+\frac{N-1}{r}\frac{u'}{\sqrt{1-(u')^2}}+|u|^{p-1}u=0
	\end{equation}
		by $u'$ and integrating on $[0,r]$ with $r<R_\xi$, we have
	\begin{equation}\label{eq:fstid}
		H(\rho(r))+(N-1)\int_0^{r} \frac{\rho(s)\O(\rho(s))}s\, ds= F(\xi)-F(u(r)),
	\end{equation}
where we have set
	\begin{equation*}
		H(\rho)=\frac {1-\sqrt{1-\rho^2}}{\sqrt{1-\rho^2}}\in C^1[0,1[.
	\end{equation*}
From identity \eqref{eq:fstid} we deduce various properties on the solution of \eqref{cauchy}. First of all, since the right hand side of equation must be positive, certainly we have $|u(r)|<\xi$ in $]0,R_\xi[$ and then, since $0<\rho(r)<1$. Moreover we clearly observe that for all $r\in[0,R_\xi[$, we have
	\begin{equation}\label{eq:impest}
        0\le \max\left( H(\rho(r)),(N-1)\int_0^{r} \frac{\rho(s)\O(\rho(s))}s\, ds\right)\le F(\xi).
    \end{equation}
    \begin{proposition}\label{pr:imp}
        For any $\xi>0$, denote by $u_\xi$ the solution of \eqref{cauchy} corresponding to $\xi>0$ and by $\rho_\xi$ the absolute value of its derivative. Then
            \begin{enumerate}
                \item there exists $a_\xi\in]0,1[$ such that
	               \begin{equation}\label{eq:boundrho}
		              \rho_\xi(r)\le a_\xi, \hbox{ in }[0,R_\xi[;
	               \end{equation}
                \item $R_\xi=+\infty$ and then $u_\xi$ is a global solution;
                \item there exists $C>0$ such that $\sup_{r>0}\rho_\xi(r)\le C\xi^{p+1}$;
                \item the integral $\int_0^{+\infty} \rho_\xi(s)\O(\rho_\xi(s))\, ds$ converges.
            \end{enumerate}
    \end{proposition}
    \begin{proof}
        1. and 4. come trivially from \eqref{eq:impest}, the definition of $H$ and the positiveness of $\O$. 2. is deduced by 1. To prove 3. we again use \eqref{eq:impest} and the definition of $F$.
    \end{proof}

As usual, an alternative formulation of equation \eqref{eq:eq} is obtained multiplying it by $r^{N-1}$ so that we have
	$$\left(r^{N-1}\frac{u'}{\sqrt{1-(u')^2}}\right)'=-r^{N-1}|u|^{p-1}u$$
and then, integrating and using initial condition on $u'$, we have
	\begin{equation}\label{eq:decr}
		\frac{u'(r)}{\sqrt{1-(u'(r))^2}}=-\frac 1 {r^{N-1}}\int_0^r s^{N-1}|u(s)|^{p-1}u(s)\, ds,
	\end{equation}
reading the previous equality in the sense of the limit when $r=0.$

Observe that, since $\xi>0$, certainly there exists a right neighborhood of $0$ where $u(r)>0$. Moreover, by \eqref{eq:decr}, we have that $u$ decreases as far as $u$ remains positive. This fact contributes to arrive to the following result
	\begin{proposition}\label{pr:either}
		A solution of \eqref{cauchy} either is a ground state, or is sign-changing.
	\end{proposition}
	\begin{proof}
		From 2. in Proposition \ref{pr:imp}, we already know that every solution of \eqref{cauchy} is global. Now, if $u$ is a positive solution, by \eqref{eq:decr} certainly it converges to some nonnegative value as $r$ goes to infinity. Calling $l$ this value and assuming $l>0$, from equation \eqref{eq:eq} and by \eqref{eq:boundrho}  we would deduce that, definitely, $u''(r)\le -\d$ for some $\d>0.$ This fact would contradict boundedness of $u'$.
	\end{proof}

So for any $u$ solution of \eqref{cauchy}, it is a useful to define by $R_0(u)$ the point where $u$ vanishes the first time, namely
	$$R_0(u):=\inf\{r>0\mid u(r)\le 0\}\in ]0,+\infty].$$

\section{Radial ground states: subcritical and supercritical case}\label{se:grst}

This section is devoted to the proofs of Theorems \ref{th:main1} and \ref{th:main2}.  Before we proceed, we need to underline the fact that the mean curvature operator satisfies the conditions $(1), (2)$ and $(3)$ in \cite{PuSe}, where $\rho$ is meant as an element in $[0,1[$ instead of $[0,+\infty[.$ The different domain of $A$ does not affect any way the results we will take from \cite{PuSe}, since, as already observed, the derivative of any solution of our equation has its absolute value away from $1$.\\
The characterizing feature in our equation is the presence of a nonlinear term $f(u)$ violating condition $(c)$ in \cite{PuSe} where it was required the existence of a number $a>0$ such that $f$ is negative in $]0,a[$ and positive in $]a,+\infty.[$

In the following subsections we first present the already cited Erbe-Tang identity in the general formulation as it appears in \cite{PuSe}, then we treat separately the subcritical and the supercritical case to prove, respectively, nonexistence and multiplicity results.

\subsection{A fundamental identity}

Let us introduce the function $P:]0,+\infty[^2\times [0,1[\to\R$
	\begin{equation}\label{eq:P}
		P(r,u,\rho)= r^N[H(\rho)+F(u)]-Nr^{N-1}\O(\rho)K(u),
	\end{equation}
where $K(u)=F(u)/f(u).$\\
The following identity, displayed in \cite[Proposition 1]{PuSe} (see also \cite{ET}), plays a key role in the proof of Theorems \ref{th:main1} and \ref{th:main2}.
	\begin{lemma}\label{le:identity}
	If $u$ is a solution of \eqref{cauchy}, then, for any $r\in]0,R_0(u)[$,
		\begin{equation*}
			 \frac{dP}{dr}(r,u(r),\rho(r))=Nr^{N-1}\rho(r)\O(\rho(r))\left\{K'(u(r))
-\frac{G(\rho(r))}{\rho(r)\O(\rho(r))}+\frac 1N\right\}.
		\end{equation*}
	\end{lemma}
		Previous identity was proved for positive, nonincreasing solutions related to a class of quasilinear equations including ours.\\
Taking into account our specific situation, the identity can be written
	\begin{equation}\label{eq:identity}
		\frac{dP}{dr}(r,u(r),\rho(r))=Nr^{N-1}\rho(r)\O(\rho(r))\left\{\frac 1 {p+1}-\frac{\sqrt{1-\rho^2(r)}}{1+\sqrt{1-\rho^2(r)}}+\frac 1 N \right\},
	\end{equation}
for any $r\in ]0,R_0(u)[$. \\
In particular, we observe that, since
	\begin{equation}\label{eq:est}
		\frac{\sqrt{1-\rho^2}}{1+\sqrt{1-\rho^2}}\le \frac 12,\quad\hbox{for any }\rho\in [0,1[
	\end{equation}
and for any $\eps>0$ there exists $\d>0$ such that
	\begin{equation}\label{eq:est2}
		\frac 12 -\eps\le \frac{\sqrt{1-\rho^2}}{1+\sqrt{1-\rho^2}},\quad\hbox{for any }\rho\in [0,\d],
	\end{equation}	
if we define
	\begin{align}\label{eq:L}
		L(u,\rho)&=K'(u)-\frac{G(\rho)}{\rho\O(\rho)}+\frac 1N\nonumber\\
                 &=\frac 1 {p+1}-\frac{\sqrt{1-\rho^2}}{1+\sqrt{1-\rho^2}}+\frac 1 N ,
	\end{align}
then we have the following result
	\begin{lemma}\label{le:estdpdr}
		The derivative $\frac{dP}{dr}$ and the function $L$ do not depend on $u$. Moreover, if $\rho\neq 0$, they have the same sign and, according to the value of $p>1$, $p\neq 2^*-1$, one or the other of the following possibilities holds:
		\begin{itemize}
			\item either $p < 2^*-1,$ and then $\frac{dP}{dr}(r,u(r),\rho(r))\ge 0$ for any 					 $r>0$,\\
			\item or $p>2^*-1,$ and then there exists $\d>0$ such that 	 $\frac{dP}{dr}(r,u(r),\rho(r))< 0$ in $\{r>0\mid \rho(r)\in ]0,\d]\}$.
		\end{itemize}
	\end{lemma}
	\begin{proof}
		Taking into account \eqref{eq:identity} and \eqref{eq:L}, we get our conclusions by \eqref{eq:est} and \eqref{eq:est2}.
	\end{proof}
	
\subsection{Case $1<p<2^*-1$: nonexistence result}

The scheme we will follow for proving Theorem \ref{th:main1} starts assuming by contradiction the existence of a radial ground state $\bar u$ such that $\bar u (0)=\bar\xi$. Then we will show that, at the same time, all solutions of \eqref{cauchy} corresponding to an initial datum $\xi\in ]0,\bar\xi[$ are ground states and there exists $\tilde \xi>0$ such that all solutions of \eqref{cauchy} corresponding to an initial datum $\xi\in]0,\tilde \xi[$ are sign-changing: obviously a contradiction.\\
The following Lemma derives from simple computations
	\begin{lemma}\label{le:comparison}
		If $u$ is a solution of \eqref{eq:eq}, then for any $\l>0$ the function $u_\l(r)=\l^{-\frac 1{2p}}u(\l^{-\frac{p-1}{4p}}r)$ solves
		\begin{equation}\label{eq:eqeps}
			 \left(\frac{w'}{\sqrt{1-\eps(w')^2}}\right)'+\frac{N-1}{r}\frac{w'}{\sqrt{1-\eps(w')^2}}+|w|^{p-1}w=0,
		\end{equation}
		with $\eps=\l^{\frac{p+1}{2p}}.$\\
Conversely, if $w$ is a solution of \eqref{eq:eqeps}, $u_\eps(r)=\eps^{\frac 1{p+1}}w( \eps^\frac{p-1}{2(p+1)} r)$ solves \eqref{eq:eq}.
	\end{lemma}
	\begin{remark}\label{re:sim}
		The relation between solutions of \eqref{eq:eq} and \eqref{eq:eqeps} was firstly exploited by Peletier and Serrin \cite{PS} to study the problem concerning the existence of a ground state solution for the prescribed mean curvature equation. We wish just to emphasize that the smaller $\eps$ is, the better equation \eqref{eq:eqeps} approximates the ODE radial formulation of \eqref{eq:LE}.
	\end{remark}
	\begin{lemma}\label{le:signch}
		If $1<p<2^*-1$, there exists $\xi_p>0$ such that for any $\xi\in]0,\xi_p]$ the solution of \eqref{cauchy} is sign-changing.
	\end{lemma}
	\begin{proof}
		Consider $v_1$, the solution to the Cauchy problem
			\begin{equation}\label{cauchyLE}
				\left\{
				\begin{array}{ll}
				v''+\frac{N-1}{r}v'+|v|^{p-1}v=0
				\\
				v'(0) =0,\\
				v(0)=1.
				\end{array}
				\right.
			\end{equation}
        It is well known that $v_1$ vanishes at a certain $\bar R$ and $v_1'(\bar R)<0$. Let $R$ be close to $\bar R$ such that $v_1(R)<0$. We set $\d=-v_1(R)/2$. Since \eqref{eq:eqeps} is a regular perturbation of the Lane-Emden equation (for $\eps=0$ the equations coincide), we can find $\bar \eps>0$ sufficiently small such that for any $\eps\in ]0,\bar\eps]$, the solution $w_\eps$ of the Cauchy problem
        	\begin{equation}\label{cauchyeps}
				\left\{
				\begin{array}{ll}
				 \left(\frac{w'}{\sqrt{1-\eps(w')^2}}\right)'+\frac{N-1}{r}\frac{w'}{\sqrt{1-\eps(w')^2}}+|w|^{p-1}w=0,
				\\
				w'(0) =0,\\
				w(0)=1,
				\end{array}
				\right.
			\end{equation}
is such that $|v_1(r)-w_\eps(r)|<\d$ in $[0,R]$.\\
Of course we deduce that every $w_\eps$ is a sign-changing solution of \eqref{eq:eqeps} which means, by Lemma \ref{le:comparison}, that $u_\eps(r)=\eps^{\frac 1{p+1}}w_\eps( \eps^\frac{p-1}{2(p+1)} r)$ is a sign-changing solution of \eqref{cauchy} with $\xi=\eps^{\frac 1 {p+1}}\in]0,\bar\eps^{\frac 1 {p+1}}]$.
        \end{proof}
We remark that the following result holds independently from the value of $p>1$, and this fact will be fundamental as we later prove Theorem \ref{th:main4} in section \ref{se:signch}.
    \begin{lemma}\label{le:nointersection}
        Assume $\bar u$ is a ground state solution of \eqref{cauchy}. Then, if $u$ is a sign-changing solution of \eqref{cauchy} such that $u(0)<\bar u(0)$, the graphs of $u$ and $\bar u$ intersect somewhere in $[0,R_0(u)]\times ]0,+\infty[$.
    \end{lemma}
    \begin{proof}
        Assume by contradiction that $u$ is a solution as in the statement and the set of points in $[0,R_0(u)]\times ]0,+\infty[$ where the graphs of $u$ and $\bar u$ intersect is empty.\\
        Set $\bar \xi=\bar u(0)$ and $\xi=u(0)$. Since $\bar u$ and $u$ are decreasing respectively in in $[0,R_0(u)]$ and in $\R^+$, we can define the functions $\bar r(u)$ and $r(u)$ inverse respectively of $\bar u$ and $u$, the first defined into $]0,\bar\xi]$, the second into $[0,\xi]$. Observe that, since $u'(R_0(u))<0$ and $\lim_{r\to+\infty}\bar u'(r)=0^-$ (the proof of this latter is the same as in \cite[pg. 146]{BLP}), then
            \begin{align*}
                \lim_{u\to 0^+}r'(u)&=\lim_{u\to 0^+}\frac{1}{u'(r(u))}=\frac 1{u'(R_0(u))}>-\infty,\\
                \lim_{u\to 0^+}\bar r'(u)&=\lim_{u\to 0^+}\frac{1}{\bar u'(\bar r(u))}=-\infty.
            \end{align*}
        Then there exists $\eta>0$ such that, if $u\in]0,\eta]$, we have $(\bar r-r)'(u)<0$.\\
        On the other hand, since $\bar u'(\bar r(\xi))<0$ and $\lim_{u\to\xi}u'(r(u))=0^-$ we have
            \begin{align*}
                \lim_{u\to \xi^-}\bar r'(u)&=\lim_{u\to \xi^-}\frac{1}{\bar u'(\bar r(u))}=\frac{1}{\bar u'(\bar r(\xi))}>-\infty,\\
                \lim_{u\to \xi^-} r'(u)&=\lim_{u\to \xi^-}\frac{1}{ u'( r(u))}=-\infty,
            \end{align*}
        and then, if $\eta>0$ is sufficiently small, we have $(\bar r-r)'(u)>0$ for $u\in[\xi-\eta,\xi[$.\\
        We deduce that, necessarily, the function $\bar r - r$ has a local minimum in the interval $]0,\xi[$. On the other hand, our contradiction assumption implies that $\bar r - r$ is positive and then, by repeating the same arguments as those in the proof of
        \cite[Lemma 3.3.1]{FLS}, it is allowed to possess at most one critical point which must be a maximum.
    \end{proof}
    \begin{lemma}\label{le:notwice}
        Let $1<p<2^*-1$ and assume $\bar u$ is a ground state solution of \eqref{cauchy}. Then, if $u$ is a sign-changing solution of \eqref{cauchy}, the graphs of $u$ and $\bar u$ can intersect in at most one point in $[0,R_0(u)]\times ]0,+\infty[$.
    \end{lemma}
    \begin{proof}
        We deduce the conclusion applying exactly the same arguments used in
        \cite[Proof of Theorem 1. Part II]{PuSe}. We just make the reader note that, even if in \cite{PuSe} the authors prove the statement for two ground states (and then both everywhere positive), their proof works fine alike if we assume that the solutions graphs intersect twice before the sign-changing solution graph touches the axis.
    \end{proof}
    Now we are ready for the following\\
    \\
    \begin{proofmain}
        Suppose by contradiction that there exists a radial solution $\bar u$ to \eqref{eq:mean} and set, with abuse of notation, $\bar u(r)=\bar u(x)$ for $|x|=r.$\\
        By Lemma \ref{le:nointersection} and Lemma \ref{le:notwice}, all solutions of \eqref{cauchy} corresponding to $\xi\in]0,\bar u(0)[$ are ground states since, otherwise, we would find a sign-changing solution violating one of the two previous lemmas.\\
        On the other hand we can find infinitely many sign-changing solutions of \eqref{cauchy} corresponding to $\xi\in ]0,\bar u(0)[$ by Lemma \ref{le:signch}.

        Of course this is a contradiction deriving from having supposed the existence of a radial ground state solution.
    \end{proofmain}

\subsection{Case $p>2^*-1$: multiplicity result}

    Since for $p$ supercritical we know that all radial solutions of \eqref{eq:LE} are ground states, the idea that multiple ground state solutions could exist for \eqref{eq:mean} arises from the fact already noted in the proof of Lemma \ref{le:signch} that, for small values of $\xi>0$, solutions of \eqref{cauchy} correspond to rescaled solutions of good approximations of problem \eqref{cauchyLE}. Of course this observation alone is not sufficient to guarantee what claimed in Theorem \ref{th:main2} since, no matter how small $\eps>0$ is, the solution coming from \eqref{cauchyeps} could, sooner or later, vanish at some $R>0$. Our proof will be based on a contradiction argument.
    \\
    \\
    \begin{proofmain2}
        Since $p>2^*-1$, by Lemma \ref{le:estdpdr} there exists $\d >0$ such that, taken any $u$ solution of \eqref{cauchy},
            \begin{equation}\label{eq:dpdr2}
                \frac{dP}{dr}(r,u(r),\rho(r))< 0 \hbox{ in }\{r>0\mid \rho(r)\in [0,\d]\}.
            \end{equation}
        Choose $\bar\xi>0$ such that, taking into account 3. of Proposition \ref{pr:imp}, it is small enough to have
            \begin{equation}\label{eq:supsup}
                \sup_{\xi\in]0,\bar\xi]}\sup_{r>0}\rho_{\xi}(r)<\d.
            \end{equation}

We claim that $(u_\xi)_{\xi\in]0,\bar\xi]}$ is a family of ground states. Suppose by contradiction that $\tilde \xi\in ]0,\bar\xi]$ is such that $u_{\tilde \xi}$ is a sign-changing solution. Then, assuming the notation $\tilde R=R(u_{\tilde\xi})$, from \eqref{eq:P} we have
	\begin{equation}\label{eq:firterm}
		P\big(0, u_{\tilde\xi}(0), \rho_{\tilde\xi} (0)\big)=0
	\end{equation}
and, since $u_{\tilde\xi}(\tilde R)=0$,
	\begin{align}\label{eq:secterm}
		&P\big(\tilde R, u_{\tilde\xi}(\tilde R), \rho_{\tilde\xi} (\tilde R)\big)\nonumber\\
		&\quad=\tilde R^N[H(\rho_{\tilde\xi}(\tilde R))+F(u_{\tilde\xi}(\tilde R))]-N\tilde R^{N-1}\O( \rho_{\tilde\xi} (\tilde R))K(u_{\tilde\xi}(\tilde R))\nonumber\\
&\quad=\tilde R^NH(\rho_{\tilde\xi}(\tilde R))\ge 0.
	\end{align}
Then, by \eqref{eq:dpdr2} , \eqref{eq:supsup},  \eqref{eq:firterm} and \eqref{eq:secterm} we achieve the following contradiction
            \begin{align*}
		0&\le P\big(\tilde R, u_{\tilde\xi}(\tilde R), \rho_{\tilde\xi} (\tilde R)\big)-P\big(0, u_{\tilde\xi}(0), \rho_{\tilde\xi} (0)\big)\\
&=\int_0^{\tilde R} \frac{dP}{dr}\big(r,u_{\tilde\xi}(r),\rho_{\tilde\xi}(r)\big)\, dr<0.
            \end{align*}
    \end{proofmain2}

\section{Asymptotic behaviour of radial ground states solutions}\label{se:dec}

In this section we want to analyze the decaying law at infinity of radial solutions of \eqref{eq:mean}. Of course, by Theorems \ref{th:main1} and \ref{th:main2}, all the contents are related with the supercritical case $p> 2^*-1$, so we do not repeat anymore this fact in the sequel.\\
It is our specific aim to compare solutions whose existence is proved by Theorem \ref{th:main2} with that in \cite[Theorem 1.1]{BDD}.

In what follows we characterize the solutions in terms of the integrability of the function $r^{N-1}\rho(r)\O(\rho(r))L(\rho(r))$.

	\begin{lemma}\label{le:Lp+1}
		If $u$ is a radial solution of \eqref{eq:mean} and $u\in L^{p+1}(\RN)$ then
			$$\int_0^{+\infty}r^{N-1}\rho(r)\O(\rho(r))L(\rho(r))\,dr=0.$$
	\end{lemma}
	\begin{proof}
		By \eqref{eq:decr},
			\begin{align}\label{eq:decu'}
				\rho(r)&\le\frac 1 {r^{N-1}}\int_0^rs^{N-1}u^p(s)\,ds\nonumber\\
					&\le \frac 1 {r^{N-1}}\left(\int_0^r   s^{N-1}u^{p+1}(s)\,ds  \right)^\frac{p}{p+1}\left(\int_0^rs^{N-1}\,ds\right)^\frac{1}{p+1}\nonumber\\
					&\le c\frac{r^\frac{N}{p+1}}{r^{N-1}}=\frac c{r^{N-1-\frac{N}{p+1}}},
			\end{align}
	where the constant $c>0$ depends on $\|u\|_{L^{p+1}}$.\\
	Since $\lim_{r\to+\infty} u(r)=0$, we have
		\begin{equation}\label{eq:decest}
			u(r)\le\int_r^{+\infty}\rho(r)\,ds\le \frac c{r^{N-2-\frac{N}{p+1}}}.
		\end{equation}	
	Then, by 1. of Proposition \ref{pr:imp}, \eqref{eq:decu'} and \eqref{eq:decest},
		\begin{align*}
			&r^N H(\rho(r))\le c r^{N}\rho^2(r)\le c r^{-N+2+\frac{2N}{p+1}}\to 0,\\
			&r^N F(u(r))\le cr^N u^{p+1}(r)\le c r^{2N-(N-2)(p+1)}\to 0,\\
			&r^{N-1}\O(\rho(r))K(u(r))\le c r^{N-1} \rho(r) u(r)\le c r^{\frac{2N}{p+1}-N+2}\to 0,
		\end{align*}
 	as $r$ goes to $+\infty.$
	At this point we refer to \eqref{eq:P}, \eqref{eq:identity} and \eqref{eq:L} to conclude as follows
		\begin{multline*}
			\int_0^{+\infty}Nr^{N-1}\rho(r)\O(\rho(r))L(\rho(r))\, dr \\
			= \lim_{r\to+\infty}r^N[H(\rho(r))+F(u(r))]-Nr^{N-1}\O(\rho(r))K(u(r))=0.
		\end{multline*}
	\end{proof}
	\begin{lemma}\label{le:integ}
		Let $u$ be a radial solution of \eqref{eq:mean}. Then the following statements are equivalent:
			\begin{itemize}
				\item[$a)$] $\int_0^{+\infty}r^{N-1}\rho(r)\O(\rho(r))L(\rho(r))\,dr\in\R$,
				\item[$b)$] $\int_0^{+\infty}r^{N-1}\rho(r)\O(\rho(r))L(\rho(r))\,dr=0$,
				\item[$c)$] $u\in\D.$
			\end{itemize}
	\end{lemma}
	\begin{proof}
		$b)\Rightarrow a)$ is obvious.\\
		Let us prove $a) \Rightarrow c)$. Since $r^{N-1}\rho(r)\O(\rho(r))\in L^1[0,+\infty]$, in view of \eqref{eq:boundrho} we deduce $\n u \in L^{2}(\RN).$ Now, since we also know that $u\in L^1_{loc}(\RN)$ (indeed $u$ is continuous) and $\meas (|u|>\a)<+\infty$ for any $\a>0$ (since $\lim_{r\to+\infty}u(r)=0$), by \cite[Remark 3]{BrLi} we have $u\in\D$.\\
$c)\Rightarrow b)$ comes from embedding $\D\hookrightarrow L^{2^*}(\RN)$ and  boundedness of $u$. Indeed we infer that $u\in L^{p+1}(\RN)$ and then we conclude by Lemma \ref{le:Lp+1}.
	\end{proof}
	\begin{remark}\label{re:alternative}
		Since we have that $\lim_{r\to+\infty}u'(r)=0$ (the proof is the same as in \cite[pg. 146]{BLP}), we deduce that definitely $L(\rho(r))<0$. As a consequence, we have that if any among $a, b$ or $c$ in the previous lemma does not hold, then, necessarily $\int_0^{+\infty}r^{N-1}\rho(r)\O(\rho(r))L(\rho(r))\,dr=-\infty$.
	\end{remark}
	
	\begin{theorem}\label{th:alternative}
		Let $u$ be a radial solution to \eqref{eq:mean}. Then
		\begin{equation*}
			u\in L^{p+1}(\RN)\iff \int_0^{+\infty}r^{N-1}\rho(r)\O(\rho(r))L(\rho(r))\,dr=0.
		\end{equation*}
	\end{theorem}
	\begin{proof}
		The conclusion comes from Lemma \ref{le:Lp+1} and Lemma \ref{le:integ}.
	\end{proof}
	
	Previous theorem allows us to conclude that solution found in \cite{BDD}, which we know is in $L^{p+1}(\RN)$, is different from any radial ground state solution found in this paper. Indeed, as showed in the proof of Theorem \ref{th:main2}, our ground states are characterized  by the fact that $L(\rho(r))< 0$ for any $r>0$ and then, of course, $\int_0^{+\infty}r^{N-1}\rho(r)\O(\rho(r))L(\rho(r))\,dr<0$.\\
In particular by Remark \ref{re:alternative}, $\int_0^{+\infty}r^{N-1}\rho(r)\O(\rho(r))L(\rho(r))\,dr=-\infty.$

An interesting question deserving some more investigation is concerned with uniqueness of radial ground state solution belonging to $L^{p+1}(\RN)$. At this time we are not able to say anything about, remaining this an open problem.

Now we proceed studying the asymptotic behaviour at infinity of radial ground state solutions.

	\begin{theorem}\label{th:fstdec}
		If $u$ is a radial solution to \eqref{eq:mean} such that $u\in L^{p+1}(\RN)$, then $u(r) = O\big(\frac 1 {r^{N-2}}\big) $.
	\end{theorem}
	\begin{proof}
		First of all, observe that by \eqref{eq:decr} and 1. of Proposition \ref{pr:imp}, we have for a suitable $c>0$
		\begin{equation*}
			\rho(r)\ge \frac c {r^{N-1}}\int_0^r s^{N-1}u^p(s)\,ds\ge \frac c {r^{N-1}}\int_0^1 s^{N-1}u^p(s)\,ds, \hbox{ for any } r>1.
		\end{equation*}
We deduce the following estimate, holding for $r>1$
	   \begin{equation*}
		  u(r)=\int_r^{+\infty} \rho(s)\,ds\ge \frac c{r^{N-2}},
	   \end{equation*}
being $c$ a positive constant depending on  $\|u\|_{L^p(B_1)}$ and on the positive constant $a$ in 1. of Proposition \ref{pr:imp}.

Now, by Lemma \ref{le:integ} and Theorem \ref{th:alternative}, we deduce that if $u$ is a radial ground state solution belonging to $L^{p+1}(\RN)$, then certainly $u\in\D.$ Thus $u\in L^{2^*}(\RN)$ and, since $u\in L^{\infty}(\RN),$ we conclude that $u\in L^q(\RN)$ for any $q\ge 2^*$. At this point there are two possibilities.

If $p\ge 2^*$, then, starting from \eqref{eq:decr}, we get the following inequality
        \begin{equation*}
            \rho(r)\le \frac 1 {r^{N-1}}\int_0^r s^{N-1}u^p(s)\,ds,
        \end{equation*}
which trivially implies that, for a suitable positive constant $c$ depending on $\|u\|_{L^{p}(\RN)}$ and for any $r>0$ we have
        $$u(r)\le \frac {c}{r^{N-2}}.$$

If $2^*-1<p<2^*$, then Holder inequality yields
        \begin{align*}
            \rho(r)&\le \frac 1 {r^{N-1}}\int_0^r s^{N-1}u^p(s)\, ds\\
                    &\le \frac 1 {r^{N-1}}\left( \int_0^r s^{N-1}\, ds\right)^{\frac {2^*-p}{2^*}}\left(\int_0^r s^{N-1}u^{2^*}\!(s)\, ds\right)^{\frac p{2^*}}\\
                    &\le \frac c{r^{N-1}}\, r^{\frac{N(2^*-p)}{2^*}}=\frac c {r^{\frac{Np}{2^*}-1}},\hbox{ for any } r>0,
        \end{align*}
from which we deduce that
        \begin{equation}\label{eq:decatinf}
            u(r)\le \frac c {r^{\frac{Np}{2^*}-2}}, \hbox{ for any } r>0,
        \end{equation}
being $c>0$ a constant depending on $\|u\|_{L^{2^*}\!(\RN)}$.\\
Observe that
        \begin{equation*}
          \frac{Np}{2^*}-2>\frac {N-2}2=\frac N{2^*}> \frac{2}{p-1},
        \end{equation*}
and then, by \eqref{eq:decatinf}, there exists $\a>\frac 2{p-1}$ such that $u(r)\le c/ r^{\a}$ for any $r>0$. By \eqref{eq:decr} we deduce that, for $r>1$,
        \begin{equation}\label{eq:fstin}
            \rho(r) \le \frac c{r^{N-1}}\left(K+\int_1^r \frac 1 {s^{\a p -N+1}}\, ds\right)
        \end{equation}
where we have set $K= \int_0^1 s^{N-1}u^p(s)\, ds$. \\
If $\a p>N,$ we easily conclude as in the case $p\ge 2^*$. So we suppose $\a p \le N$ and computing
the integral in \eqref{eq:fstin}, we get
        \begin{equation*}
            \rho(r) \le \frac c {r^{\a p -1}}, \hbox{ for any } r>0.
        \end{equation*}
We deduce that for any $r>0$
        \begin{equation*}
            u(r)\le \frac c{r^{\a p -2}}.
        \end{equation*}
Observe that, since $\a> N/2^*$, the decay estimate we have obtained improves \eqref{eq:decatinf}. \\
Now, repeating the computations made below with $\a p -2$ in the place of $\a$, we again achieve easily our conclusion if $(\a p -2)p >N,$ otherwise, as before, we get a new improved decay estimate as follows
        \begin{equation*}
            u(r)\le \frac c{r^{(\a p-2) p -2}}=\frac c {r^{\a p^2-2p-2}}, \hbox{ for any } r>0.
        \end{equation*}
Iterating, at the $n$-th step we have
        \begin{equation*}
            u(r)\le \frac c{r^{\a p^n-2 p^{n-1}-\ldots -2}}, \hbox{ for any } r>0.
        \end{equation*}
Computing we observe that
        \begin{equation*}
            \a p^n-2 p^{n-1}-\ldots -2=\a p^n -2 \sum_{k=0}^{n-1} p^k=\a p^n+\frac 2{p-1}(1-p^n)
        \end{equation*}
and then, since $\a>2/(p-1)$ and $p>1$, the sum diverges positively and must achieve a value $\beta$ exceeding $\frac N p$ in a finite number of steps. When this happens, we will deduce our conclusion using inequality \eqref{eq:fstin} with $\beta$ in the place of $\a$.
	\end{proof}

    \begin{theorem}\label{th:dec2}
        Suppose $u$ is a radial solution of \eqref{eq:mean} such that $u\notin L^{p+1}(\RN)$. Then
        there exist $c_1>0$ and $c_2>0$ such that $c_1/r^{\frac{2N}{(N-1)(p+1)-2N}}\le u(r)\le c_2/r^{\frac 2{p-1}}$, for any $r>1.$\\
        Moreover there exists no $\a>2/(p-1)$ such that, definitely, $u(r)\le c/r^\a$ for some $c>0.$
    \end{theorem}
    \begin{proof}
        By Theorem \ref{th:alternative} and Lemma \ref{le:identity}, we have
            \begin{equation*}
                \lim_{r\to+\infty} \big\{r^N[H(\rho(r))+F(u(r))]- N r^{N-1}\O(\rho(r))K(u(r))\big\}=-\infty.
            \end{equation*}
        In particular, definitely, we have
            \begin{equation}\label{eq:definitely}
                H(\rho(r))+F(u(r))\le \frac N r \O(\rho(r))K(u(r)),
            \end{equation}
and then, by the definitions of $H, F, \O$ and $K$, we deduce (coming estimates are to be understood definitely for $r$ large)
	\begin{align*}
		1-\sqrt{1-\rho^2(r)}&\le \frac N {p+1} \frac{\rho(r) u(r)}{r}\\
		u^p(r)&\le\frac N r \frac{\rho(r)}{\sqrt{1-\rho^2(r)}}.
	\end{align*}
Now, since $\frac 1 2 \rho^2(r)\le 1- \sqrt{1-\rho^2(r)},$ by the first of the previous inequalities we have
	$$\frac 1 2 \rho(r)\le \frac N {p+1}\frac {u(r)}{r},$$
and, comparing with the second one,
		$$u^p(r)\le \frac {2N^2}{p+1}\frac {u(r)}{r^2\sqrt{1-\rho^2(r)}}.$$
By 1. of Proposition \ref{pr:imp}, we conclude that for some $c_2>0$, we have $u(r)\le c_2/r^{\frac{2}{p-1}}$.

Now we proceed with the below estimate, again assuming that the inequalities we obtain hold definitely for large $r$. Deriving in \eqref{eq:fstid}, we have
	\begin{equation*}
		[H(\rho(r))+F(u(r))]'=-(N-1)\frac{\rho(r)\O(\rho(r))}r,
	\end{equation*}
which, integrated in $(r,+\infty)$, gives
    \begin{equation*}
        [H(\rho(r))+F(u(r))]=(N-1)\int_r^{+\infty}\frac{\rho(s)\O(\rho(s))}s\,ds.
    \end{equation*}
Comparing with \eqref{eq:definitely}, and taking into account the definition of $K$, we have
	\begin{equation}\label{eq:intineq}
		(N-1)\int_r^{+\infty}\frac{\rho(s)\O(\rho(s))}{s}\, ds\le \frac{N}{p+1}\frac{\O(\rho(r))u(r)}{r},
	\end{equation}
where we have used also the fact that $\lim_{r\to +\infty}H(\rho(r))+F(u(r))=0$.
    \end{proof}
Now, if we set $h(r)=\int_r^{+\infty}\frac{\rho(s)\O(\rho(s))}{s}\, ds$ we transform \eqref{eq:intineq} in the following differential inequality
	\begin{equation*}
		\frac \rho u \le -\frac N{(N-1)(p+1)}\frac{h'(r)}{h(r)}.
	\end{equation*}
Then, integrating in $[R,r]$, we have the following inequalities, holding for a suitable $c>0$
	\begin{align*}
		u(r)&\ge c[h(r)]^{\frac N{(N-1)(p+1)}}\nonumber=c\left[\frac  {H(\rho(r))+F(u(r))}{N-1}\right]^{\frac N{(N-1)(p+1)}}\nonumber\\
			&\ge c [H(\rho(r))]^{\frac N{(N-1)(p+1)}}\nonumber\\
			&\ge  c\rho^{\frac{2N}{(N-1)(p+1)}}
	\end{align*}
that is, for a suitable $c>0$,
	\begin{equation*}
		-\frac{u'}{u^\frac{(N-1)(p+1)}{2N}}\le c.
	\end{equation*}
Integrating in $[R,r]$, we have that there exists $c>0$ such that
	\begin{equation*}
		\left(\frac{(N-1)(p+1)}{2N}-1\right)\left(\frac 1 {u(r)}\right)^{\frac{(N-1)(p+1)}{2N}-1}\le c r.
	\end{equation*}
Therefore we have
	\begin{equation*}
		(u(r))^{\frac{(N-1)(p+1)-2N}{2N}}\ge \frac c r
	\end{equation*}
which leads to our conclusion.

The final sentence derives directly from the arguments developed in the proof of Theorem \ref{th:fstdec}, assuming by contradiction the existence of $\a > \frac 2{p-1}$ such that, definitely, $u(r)	\le c/r^\a$.
\\
\\
{\it Proof of Theorem \ref{th:main3}.}$\quad$It is a consequence of Theorems \ref{th:fstdec} and \ref{th:dec2}.

\section{Radial sign-changing bound states}\label{se:signch}

We recall that, with {\it bound states}, we mean solutions going to zero as $r$ goes to $+\infty$. This section is devoted to showing that,
all radial sign-changing solutions are bound state. Now, since all radial solutions of our mean curvature equation are sign-changing when $p$ is taken subcritical, we would provide a multiplicity result characterizing any solution for $p<2^*-1$. As regards supercritical case, a first important result, actually holding for any $p>1$,  is the following

	\begin{theorem}\label{th:signch}
		If $p>1$, then there exists $\tilde \xi>0$ such that any solution of \eqref{cauchy} corresponding to $\xi\ge\tilde\xi$ is sign-changing.
	\end{theorem}
	\begin{proof}
		By \cite[Example 5]{BJT}, we know that there exists a sufficiently large $\tilde R>0$ such that problem
	\begin{equation*}
\left\{
\begin{array}{ll}
\left(\frac{u'}{\sqrt{1-(u')^2}}\right)'+\frac{N-1}{r}\frac{u'}{\sqrt{1-(u')^2}}+u^p=0,
\\
u'(0) =0,\\
u(\tilde R)=0,
\end{array}
\right.
\end{equation*}
possesses a positive solution.\\
Now, we set $\tilde \xi=\tilde R$ and pick any $\xi\ge\tilde\xi$. By  1. of Proposition \ref{pr:imp}, we easily deduce that the graph of $u_\xi$ never intersects that of $u_{\tilde\xi}$ in $[0,\tilde R]\times ]0,+\infty[$. Then, since Lemma \ref{le:nointersection} states that $u_\xi$ can not be a ground state, by Proposition \ref{pr:either} it is sign-changing.
	\end{proof}

Now we study the behaviour at infinity of sign-changing solutions.

 	\begin{theorem}\label{th:to0}
		If $u=u_\xi$ is solution of \eqref{cauchy}, then it is global and $\displaystyle\lim_{r\to +\infty}u(r)=0$.
	\end{theorem}
	\begin{proof}
		By 2. of Proposition \ref{pr:imp}, certainly all solutions are global. By Proposition \ref{pr:either}, we only have to prove the following claim: every sign-changing solution of \eqref{cauchy} originates a solution to \eqref{eq:meanpm}.\\
So, let $u$ be a sign-changing solution and $\A:=\{r>0\mid u(r)=0\}$. We distinguish the two possibilities and show that, in any case, $\lim_{r\to+\infty}u(r)=0$.
    \begin{description}
        \item[{\it 1st case: }]$\exists \max\A=\hat R$.

        Suppose, to fix ideas, that $u(r)<0$ in $]\hat R,+\infty[$ and, consequently, $u'(\hat R)<0$.

If the sign of $u'$ does not change anymore, we have
	$$u(r)\searrow k\in [-\infty, 0[.$$
Actually $k\neq-\infty$ since $u$ is bounded  (otherwise in \eqref{eq:fstid} the difference $F(\xi)- F(u(r))$ becomes somewhere negative) and then, using \eqref{eq:eq} and recalling \eqref{eq:boundrho}, we have $\lim_{r\to+\infty}u''(r)>0$, and then $\lim_{r\to+\infty}u'(r)=+\infty$, obviously a contradiction.

If $u'(R)=0$ at some $R>\hat R$, then for any $r>R$ we have $u'(r)>0$ since no more critical point can be present at the right of $R$ (otherwise, since $u$ is definitely negative, by \eqref{eq:eq} at this point we should have minimum and this is impossible).\\
We again deduce that there exists $k=\lim_{r\to+\infty} u(r)\le 0$ and, supposing $k<0$, we achieve contradiction as before.

\item[{\it 2nd case: }]$\not\!\exists\max\A$.

Set $M=(N-1)\int_0^{+\infty}\frac{\rho^2(s)}{s\sqrt{1-\rho^2(s)}}\, ds$ which is in $\R$ by 4. of Proposition \ref{pr:imp}. If $M=\frac{\xi^{p+1}}{p+1}$ then we conclude by \eqref{eq:fstid}. Suppose by contradiction $M<\frac{\xi^{p+1}}{p+1}$ and set $\eps=k\left(\frac{\xi^{p+1}}{p+1}-M\right)$ where $k>0$ is such that, taking into account  \eqref{eq:boundrho} and the fact that $H(\rho(r))=O(\rho^2(r))$ for $r\to 0^+$, for any $r>0$ we have
	\begin{equation}\label{eq:rho^2rho}
		\rho(r)>k H(\rho(r)).
	\end{equation}
Now, if $\bar r \in\A$, since $\rho(r)$ is bounded, there exists $\eta>0$ not depending on the choice of $\bar r$ such that
	\begin{equation*}
		\frac 1 {p+1}|u(r)|^{p+1}< \frac \eps {2k}\quad\hbox{in }\left[\bar r -\frac \eta 2, \bar r+\frac \eta 2\right].
	\end{equation*}
Then, by \eqref{eq:fstid}, for any $r\in \left[\bar r -\frac \eta 2, \bar r+\frac \eta 2\right]$ we have
	\begin{equation}\label{eq:rhogreater}
		\rho(r)>k\left(\frac 1 {p+1} \xi^{p+1}- M -\frac \eps {2k}\right)=\frac\eps 2.
	\end{equation}
Observe that, by \eqref{eq:rhogreater}, $\A$ can be organized as an increasing divergent sequence $(r_n)_{n\ge 1}$.

Now we would obtain a uniform superior estimate for the distance between two consecutive points in $\A$. Call $R_1$ and $R_2$ two elements in $\A$, with $R_1<R_2$ and, to fix ideas, suppose $u(r)>0$ in $]R_1,R_2[$.\\
Set $\a>0$ such that $\frac 1{p+1}(\xi^{p+1}-\a^{p+1})-M>0,$ $\g\in]0,\a[$ and $\tilde R >0$ such that for any $r>\tilde R$ we have
	\begin{equation*}
		\frac {N-1}r\frac {\rho(r)}{\sqrt{1-\rho^2(r)}}<\g.
	\end{equation*}

		 If $u^p(r)\ge\a$, by \eqref{eq:eq} we have
	\begin{equation*}
		\frac{u''(r)}{\sqrt{(1-\rho^2(r))^3}}\le- \a+\g<0,\hbox{ for any } r>\tilde R,
	\end{equation*}
if $u^p(r)<\a$, by \eqref{eq:fstid} we have
	\begin{equation*}
		H(\rho(r))\ge \frac 1{p+1}(\xi^{p+1}-\a^{p+1})-M>0,\hbox{ for any } r>0.
	\end{equation*}
By these computations, taking into account \eqref{eq:boundrho} and \eqref{eq:rho^2rho}, we conclude that there exists $\d>0$ such that, for any $r>\tilde R$, we have
	\begin{itemize}
		\item   $u''(r)\le -\d$ if $u^p(r)\ge \a$,
		\item   $\rho(r)\ge \d$ if $u^p(r)<\a$.
	\end{itemize}
Starting from $R_1$, by the second estimate we deduce that the largest possible interval before $u$ passes the level $\sqrt[p]{\a}$ is $[R_1, R_1+\sqrt[p]\a/\d[$.\\
By the second estimate, considering the parabola $\G$ having second derivative $-\d$ and passing through the point $(\sqrt[p]\a/\d,\sqrt[p]\a)$ with velocity $1$, we can claim that certainly the graph of $u$ touches a second time the line $u=\sqrt[p]\a$ before the parabola $\G$ does. Then the largeness of the interval $\{r\in]R_1,R_2[\mid u(r)\ge \sqrt[p]\a\}$ is less than $2/\d$.\\
Finally, using again the second estimate in the same way as before, we deduce that the largeness of the second connected component of $\{r\in]R_1,R_2[\mid u(r)< \sqrt[p]\a\}$ is less than $\sqrt[p]\a/\d$.

So, if we set $\sigma=\frac 2 \d(1+\sqrt[p]\a)$, we have that, for any $n\ge n_0$ with $n_0$ large enough,
    \begin{equation}\label{eq:boundint}
        r_{n+1}-r_n\le \sigma.
    \end{equation}

By \eqref{eq:rhogreater} and \eqref{eq:boundint},
    \begin{align*}
        \frac M{N-1}&= \int_0^{+\infty}\frac{\rho^2(s)}{s\sqrt{1-\rho^2(s)}}\, ds\\
                    &\ge \lim_n\sum_{n\ge n_0}\int_{r_n-\frac \eta 2}^{r_n+\frac\eta 2}\frac{\rho^2(s)}{s\sqrt{1-\rho^2(s)}}\, ds\\
                    &\ge \lim_n\frac {\eta\eps^2}{2\sqrt{4-\eps^2}}\sum_{n\ge n_0}\frac 1 {r_n+\eta/2}\\
                    &\ge \lim_n\frac {\eta\eps^2}{2\sqrt{4-\eps^2}}\sum_{n\ge 0}\frac 1 {r_{n_0}+n\sigma+\eta/2}=+\infty.
    \end{align*}
This is obviously a contradiction by which we conclude.
\end{description}

\end{proof}
{\it Proof of Theorem \ref{th:main4}.}$\quad$It is a consequence of Theorems \ref{th:signch} and \ref{th:to0}.

\end{document}